\documentclass{siamart}   	
\usepackage{lipsum}
\usepackage{amsfonts}
\usepackage{nccmath}
\usepackage{graphicx}
\usepackage{epstopdf}
\usepackage{algorithmic}
\ifpdf
  \DeclareGraphicsExtensions{.eps,.pdf,.png,.jpg}
\else
  \DeclareGraphicsExtensions{.eps}
\fi

\newcommand{\TheTitle}{Disjoint Spread Systems and Fault Location}
\newcommand{\TheAuthors}{Charles J. Colbourn, Bingli Fan, and Daniel Horsley}

\headers{\TheTitle}{\TheAuthors}

\title{{\TheTitle}\thanks{The first author's research was supported in part by the National Science Foundation under Grant No. 1421058.
The third author's research was supported by Australian Research Council grants DE120100040 and DP150100506.}}

\author{
  Charles J. Colbourn\thanks{CIDSE,  Arizona State University,
  Tempe AZ 85287-8809,
  U.S.A.  (\email{colbourn@asu.edu})}
  \and
  Bingli Fan\thanks{Mathematics,
Beijing Jiaotong University,
Beijing, China (\email{blfan@bjtu.edu.cn})}
\and
  Daniel Horsley\thanks{Mathematical Sciences,
Monash University,
Vic 3800, Australia (\email{daniel.horsley@monash.edu})}}

\usepackage{amsopn}

\usepackage{amssymb}
\def \mod#1{{\:({\rm mod}\ #1)}}
\def \pmod#1{{\:({\rm mod}\ #1)}}

\begin{document}
\maketitle
\begin{abstract}
When $k$ factors each taking one of $v$ levels may affect the correctness or performance of a complex system, a test is selected by setting each factor to one of its levels and determining whether the system functions as expected (passes the test) or not (fails).
In our setting, each test failure can be attributed to at least one faulty (factor, level) pair.
A nonadaptive test suite is a selection of such tests to be executed in parallel.
One goal is to minimize the number of tests in a test suite from which we can determine which (factor, level) pairs are faulty, if any.
In this paper, we determine the number of tests needed to locate faults when exactly one (or at most one) pair is faulty.
To do this, we address an equivalent problem, to determine how many set partitions of a set of size $N$  exist in which each partition contains $v$ classes and no two classes in the partitions are equal.
\end{abstract}

\begin{keywords} Locating array, covering array, Sperner partition system, Baranyai's theorem. \end{keywords}
\begin{AMS} 05B30 (primary), 05A18, 05D99, 62K05, 68P10 (secondary) \end{AMS}

\section{Introduction}

Determining the presence and location of faults in complex systems encompasses a wide variety of problems.
When many factors affecting correctness or performance are present, each having a variety of possible options or levels, combinatorial test suites have been widely studied to reveal the presence of faults arising from interactions that result from a small set of factors being set to specific levels \cite{GOA,Hartman,NieL-CS}.
Combinatorial test suites that guarantee that every such interaction appears in a test are known as covering arrays.
Covering arrays provide a method to reveal the presence of faults  \cite{croatia,Hartman}, but they are inadequate to determine which interaction(s) account for the faulty behaviour.
 Colbourn and McClary \cite{CMjoco} extended  covering arrays to provide sufficient information to identify all faults when few faults, each involving few factors, are present.

Formally, there are $k$ factors $F_1, \dots, F_k$.
Each factor $F_i$ has a set  of $s_i$ possible values ({\em levels}) $S_i = \{v_{i1}, \dots, v_{i{s_i}}\}$.
A {\em test} is an assignment, for each $i$ with $1 \leq i \leq k$, of a level from $v_{i1}, \dots, v_{i{s_i}}$ to $F_i$.
A test, when executed, can {\em pass} or {\em fail}.
For any $t$-subset $I \subseteq \{1,\dots,k\}$ and levels $\sigma_{i} \in S_{i}$ for $i \in I$, the set $\{(i,\sigma_{i}) : i \in I\}$ is a {\em $t$-way interaction}, or an interaction of {\em strength} $t$.
Thus a test on $k$ factors contains  ({\em covers}) $\binom{k}{t}$ interactions of strength $t$.
A {\em test suite} is a collection of tests; the {\em outcomes} are the corresponding set of pass/fail results.
A fault is evidenced by a failure outcome for a test.
A fault is rarely due to a complete $k$-way interaction; rather it is the result of one or more faulty interactions of strength smaller than $k$ covered in the test.
Tests are executed concurrently, so that testing is {\em nonadaptive} or {\em predetermined}.

We employ a matrix representation.
An array $A$ with $N$ rows, $k$ columns, and symbols in the $i$th column chosen from an alphabet $S_i$ of size $s_i$ is denoted as an $N \times k$ array of symbol type $(s_1,\dots,s_k)$.
A {\em $t$-way interaction} in $A$ is a choice of a set $I$ of $t$ columns, and the selection of a level $\sigma_{i} \in S_{i}$ for $i \in I$, represented as $T = \{(i,\sigma_{i}) : i \in I\}$.
For such an array $A=(a_{xy})$ and interaction $T$, define $\rho_A(T) = \{ r : a_{ri} = \sigma_{i} \mbox{ for each } i \in I\}$, the set of rows of $A$ in which the interaction is covered.
For a set of interactions $\cal T$,  $\rho_A({\cal T}) = \bigcup_{T \in {\cal T}} \rho_A(T)$.

Let ${\cal I}_t$ be the set of {\em all} $t$-way interactions for an array of symbol type $(s_1,\dots,s_k)$, and let $\overline{{\cal I}_t}$ be the set of all interactions of strength {\em at most} $t$. Consider an interaction $T \in \overline{{\cal I}_t}$ of strength less than $t$.
Any interaction $T'$ of strength $t$ that contains $T$ necessarily has $\rho_A(T') \subseteq \rho_A(T)$;  a subset ${\cal T}'$ of interactions in ${\cal I}_t$ is {\em independent} if there do not exist $T, T' \in {\cal T}'$ with $T \subset T'$.
Some  interactions are believed to cause faults.  To formulate arrays for testing, we assume limits on both the number of interactions causing faults and their strengths.

\renewcommand{\tabcolsep}{6pt}
As in \cite{CMjoco}, this leads to a variety of arrays $A$ for testing a system with $N$ tests and  $k$ factors having $(s_1,\dots,s_k)$ as the numbers of levels, defined in Table \ref{defn:faultarray}.
It defines {\em mixed covering arrays} (MCAs), {\em covering arrays} (CAs), {\em locating arrays} (LAs), and {\em detecting arrays} (DAs).
When all factors have the same number of levels $v$, we replace $(s_1,\dots,s_k)$ with $v$ in the notation.

\begin{table}[ht]
\begin{center}
\begin{tabular}{|p{1.7in}|p{3.0in}|}
\hline
Array & Definition \\
\hline
\hline
{\bf Covering Arrays:} \rule{0mm}{3.5mm}& $\rho_A(T) \neq \emptyset$ for all $T \in \overline{{\cal I}_t}$ \\ \hline
MCA$(N;t,k,(s_1,\dots,s_k))$ & \\
CA$(N;t,k,v)$ & also $v = s_1 = \cdots = s_k$\\
\hline
\hline
{\bf Locating Arrays:} \rule{0mm}{3.5mm}& $\rho_A({\cal T}_1) = \rho_A({\cal T}_2)  \Leftrightarrow {\cal T}_1 = {\cal T}_2$ whenever... \\\hline
 $(d,t)$-LA$(N;k,(s_1,\dots,s_k))$ &  ${\cal T}_1, {\cal T}_2 \subseteq {{\cal I}_t}$,  $|{\cal T}_1| = d$, and $|{\cal T}_2| = d$\\
$(\overline{d},t)$-LA$(N;k,(s_1,\dots,s_k))$ &  ${\cal T}_1, {\cal T}_2 \subseteq {{\cal I}_t}$,  $|{\cal T}_1| \leq d$, and $|{\cal T}_2| \leq d$\\
 $(d,\overline{t})$-LA$(N;k,(s_1,\dots,s_k))$ &  ${\cal T}_1, {\cal T}_2 \subseteq {\overline{{\cal I}_t}}$,  $|{\cal T}_1| = d$,  $|{\cal T}_2| = d$, and ${\cal T}_1$ and  ${\cal T}_2$ are independent\\
$(\overline{d},\overline{t})$-LA$(N;k,(s_1,\dots,s_k))$ & ${\cal T}_1, {\cal T}_2 \subseteq {\overline{{\cal I}_t}}$,  $|{\cal T}_1| \leq d$,  $|{\cal T}_2| \leq d$, and ${\cal T}_1$ and  ${\cal T}_2$ are independent\\
\hline
\hline
{\bf Detecting Arrays:} \rule{0mm}{3.5mm}& $\rho_A(T) \subseteq \rho_A({\cal T})  \Leftrightarrow T \in  {\cal T}$ whenever...\\ \hline
 $(d,t)$-DA$(N;k,(s_1,\dots,s_k))$ & $T \in {{\cal I}_t}$, ${\cal T} \subseteq {{\cal I}_t}$,  and $|{\cal T}| \leq d$\\
 $(d,\overline{t})$-DA$(N;k,(s_1,\dots,s_k))$ & $T \in \overline{{\cal I}_t}$, ${\cal T} \subseteq \overline{{\cal I}_t}$,   $|{\cal T}| \leq d$, and ${\cal T}\cup\{T\}$ is independent\\
\hline
\end{tabular}
\end{center}
\caption{Arrays for determining faults}\label{defn:faultarray}
\end{table}

Detecting arrays may permit faster recovery than locating arrays because faulty interactions can be found by simply listing all interactions that appear only within failed tests.
Locating arrays have been utilized in applications to measurement and testing \cite{Aldaco15}, but few constructions are known.
Mart\'inez {\sl et al.} \cite{Martinez10} develop adaptive analogues and establish feasibility conditions for a locating array to exist.
In \cite{Shi12} and \cite{TangCY12} the minimum number of rows in a locating array is determined when the number of factors is quite small.
Recursive constructions when $(d,t)=(1,2)$ are given in \cite{CFan}.

Even when $(d,t)=(1,1)$, so that we are locating one faulty level of one factor, the smallest locating and detecting arrays are not known.
In this paper, we focus on locating arrays of strength  1.
We determine the smallest number of rows $N$ in a $(1,1)$-LA$(N;k,v)$.
To do this, we treat the equivalent problem of determining, for given $N$ and $v$, the largest $k$ for which a $(1,1)$-LA$(N;k,v)$ exists.
This maximum is denoted ${\rm LAK}_{(1,1)}(N,v)$.

Because we focus on the case $(d,t)=(1,1)$, we are locating one faulty member in a finite population (consisting of the (factor, level) choices).
Hence our problem belongs to the province of combinatorial search \cite{Aigner88}, search theory  \cite{Ahlswede87}, or combinatorial group testing \cite{DuHwang}.
These are well-researched topics, and a complete treatment is not possible here.
Among variants closest to the one that we examine, R\'enyi \cite{Renyi61} and Katona \cite{Katona66} consider tests in which elements are partitioned into $v$ classes, and the tester is told which class, if any, contains the faulty item; for more recent work, see \cite{Baranyai78a,DeBonis98,Meng12}.
Unlike our situation, these locate a factor, not a factor and a level; and each test has $v$ outcomes, not two.

It appears that the search problem arising from locating arrays of strength one, while in some ways similar to problems in the literature, has not been studied.
We can reformulate questions about certain covering, detecting, and locating arrays as follows:

\begin{proposition}\label{turntosets}\mbox{}
\begin{enumerate}
\item A $CA(N;2,k,v)$ is equivalent to a set of $k$ partitions of $\{1,\dots,N\}$, each partition having $v$ classes,  such that each class in each partition has non-empty intersection with every class of every other partition.
\item A $(1,1)$-detecting array on $v$ symbols with $N$ rows and $k$ columns is equivalent to a set of $k$ partitions of $\{1,\dots,N\}$, each partition having $v$ classes,  such that none of the $kv$ classes is a subset of another.
\item A $(1,1)$-locating array on $v$ symbols with $N$ rows and $k$ columns is equivalent to a set of $k$ partitions of $\{1,\dots,N\}$, each partition having $v$ (possibly empty) classes, such that no two of the $kv$ classes are equal.
\item A $(\bar{1},1)$-locating array on $v$ symbols with $N$ rows and $k$ columns is equivalent to a set of $k$ partitions of $\{1,\dots,N\}$, each partition having $v$ classes, such that none of the $kv$ classes is empty and no two of the $kv$ classes are equal.
\end{enumerate}
\end{proposition}
\begin{proof}
Each column of an $N \times k$ array on $v$ symbols forms a partition of the row indices into $v$ classes.
For column $\gamma$ and symbol $\sigma$, $\rho(\{(\gamma,\sigma)\})$ is a class of the partition formed from column $\gamma$.
The statement that each class in each partition has non-empty intersection with every class of every other partition is equivalent to the requirement that all 2-way interactions be covered.
The statement that none of the $kv$ classes is a subset of another is the same as the statement that $\rho(\{(\gamma,\sigma)\}) \nsubseteq \rho(\{(\gamma',\sigma')\})$ when $(\gamma,\sigma) \neq (\gamma',\sigma')$. The statement that no two of the $kv$ classes are equal is the same as the statement that $\rho(\{(\gamma,\sigma)\}) \neq \rho(\{(\gamma',\sigma')\})$ when $(\gamma,\sigma) \neq (\gamma',\sigma')$. The statement that none of the $kv$ classes is empty is the same as the statement that $\rho(\{(\gamma,\sigma)\}) \neq \emptyset = \rho(\emptyset) $.
\end{proof}

Two cases in Proposition \ref{turntosets} have been studied.
The partitions from  $CA(N;2,k,v)$s have been studied as ``qualitatively independent partitions'' \cite{renyi}.
Given $N$,  the largest $k$ for which a $CA(N;2,k,v)$ exists is known only when $v=2$ \cite{katona,kleitman}, although useful bounds are known for all $v \geq 3$ \cite{cmmssy}.
By Proposition \ref{turntosets}, the (1,1)-detecting arrays are equivalent to the so-called  Sperner partition systems \cite{MeagherLi13,MeagherMS05}; the maximum $k$ for specified $N$ is determined only within certain congruence classes for $N$ modulo $v$.
Qualitatively independent partitions and Sperner partition systems yield (1,1)-locating arrays, but in general permit far fewer factors.
The problem for (1,1)-locating arrays appears to be more tractable, because only distinctness of the classes  is required.
Indeed in the remainder of the paper, we provide a complete solution for existence for (1,1)-, $(\bar{1},1)$-, $(1,\bar{1})$-, and $(\bar{1},\bar{1})$-locating arrays.
Our main result is as follows.

\begin{theorem}\label{mainThm}
Let $N$ and $v$ be integers such that $2 \leq v \leq N+1$.
Let $f = \lfloor \frac{N+1}{v} \rfloor$, $d=(f+1)v-N$, and
\[\Lambda(N,v)=\big\lfloor\tfrac{1}{d}\textstyle{\sum_{i=f-d+2}^{f}(f+1-i)}\tbinom{N}{i}\big\rfloor+\textstyle{\sum_{i=0}^{f-d+1}}\tbinom{N}{i} . \]
Then ${\rm LAK}_{(1,1)}(N,v)=\Lambda(N,v)$.
\end{theorem}
Observe that if $v \geq N+2$, then a $(1,1)$-LA$(N;k,v)$ cannot exist for $k > 0$ because no partition of $\{1,\ldots,N\}$ into $v$ distinct classes exists.

An overview of the remainder of the paper follows.
In Section \ref{sec:bar} we generalize a theorem of Baranyai, in order to reduce the construction of locating arrays on $N$ rows to the existence of certain sets of integer partitions of $N$.
In Section \ref{sec:types} we specify precisely the integer partitions employed to realize the bound.
In order to establish that these integer partitions are indeed feasible, in Section \ref{sec:bin} we derive some technical lemmas giving inequalities on sums of binomial coefficients.
Then in Section \ref{sec:main} we use the binomial inequalities to establish that the partitions of Section \ref{sec:types} are feasible, thereby completing the proof of Theorem \ref{mainThm}.
In Section \ref{sec:variant} we give a complete solution for  $(\bar{1},1)$-locating arrays (Theorem \ref{1barThm}) and for $(1,\bar{1})$- and $(\bar{1},\bar{1})$-locating arrays (Theorem \ref{barvar}).
We conclude in Section \ref{sec:asym} with a brief discussion of the asymptotic differences between (1,1)-locating arrays and covering arrays of strength two.

\section{Partial Spreads and Disjoint Spread Systems}\label{sec:bar}

Let $X$ be a set of size $N$.
Let ${\mathcal P}(X)$ be the set of all subsets of $X$, the {\em powerset} of $X$.
Let ${\mathcal S} \subseteq {\mathcal P}(X)$; sets in ${\mathcal S}$ are termed {\em blocks}.
A {\em partial spread} in ${\mathcal S}$ is a subset ${\mathcal C} \subseteq {\mathcal S}$ so that whenever $C,C' \in {\mathcal C}$, either $C \cap C' = \emptyset$ or $C = C'$.
It is a {\em spread} when, in addition, $\bigcup_{C \in {\mathcal C}} C = X$.
When a partial spread contains precisely $v$ blocks, it is a $v$-{\em partial spread}.

Two partial spreads ${\mathcal C}_1$ and ${\mathcal C}_2$ are {\em disjoint} if whenever $C_1 \in {\mathcal C}_1$ and $C_2 \in {\mathcal C}_2$, we have $C_1 \neq C_2$.
A partition of ${\mathcal S}$ into  partial spreads $\{ {\mathcal C}_i : 1 \leq i \leq k\}$ is a {\em disjoint partial spread system;} when each partial spread is a spread, it is a {\em disjoint spread system.}

A {\em shape} is a multiset of nonnegative integers $\{a_1,\dots,a_\ell\}$ for which $\sum_{i=1}^\ell a_i \leq N$; it is a {\em $v$-shape} if it contains exactly $v$ entries. For a shape $S$, we denote the number of entries of $S$ that are equal to $x$ by $\mu_S(x)$. A {\em type} is a multiset of shapes; it is a {\em $v$-type} if every shape in it is a $v$-shape. Any partial spread ${\mathcal C}$ has shape $\{ |C| : C \in {\mathcal C}\}$.
A disjoint partial spread system has a type that consists of the shapes of its partial spreads.

Baranyai \cite{Baranyai} proved a remarkable theorem, implying that when $|X| = uv$ and ${\mathcal S}$ consists of all subsets of $X$ of size $u$, there always exists a disjoint spread system on ${\mathcal S}$; each spread in the system is an $v$-spread.
In order to address existence of locating arrays, we treat a more general situation.
To formulate the question precisely, fix a set $X = \{1,\dots , N\}$.
Let $k$ be a nonnegative integer.
For $1 \leq i \leq k$, let $M_i$ be a shape.
We are to determine whether there exists a disjoint partial spread system on a set ${\mathcal S} \subseteq {\mathcal P}(X)$ having  type $\{ M_1 , \dots, M_k\}$.

A type $\{ M_1, \dots, M_k\}$ is {\em realizable} when there is a disjoint partial spread system of that type on some ${\mathcal S} \subseteq {\mathcal P}(X)$.
A type $\{ M_1, \dots, M_k\}$ is {\em admissible} when
\[\medop \sum_{i=1}^k \mu_{M_i}(x) \leq \mbinom{N}{x} \quad \mbox{for } 0 \leq x \leq N,\]
and is {\em full} when we have equality in each of these inequalities. In a realizable type, the inequalities hold because no set in ${\mathcal P}(X)$ can appear in more than one partial spread.

\begin{proposition}\label{typeequiv}\mbox{}
\begin{enumerate}
    \item
There exists a $(1,1)$-locating array on $v$ symbols with $N$ rows and $k$ columns if and only if there exists a realizable $v$-type on $\{1,\ldots,N\}$ consisting of $k$ shapes.
    \item
There exists a $(\bar{1},1)$-locating array on $v$ symbols with $N$ rows and $k$ columns if and only if there exists a realizable $v$-type on $\{1,\ldots,N\}$ consisting of $k$ shapes such that no shape contains a $0$.
\end{enumerate}
\end{proposition}
\begin{proof}\mbox{}
\begin{enumerate}
    \item
By Proposition \ref{turntosets}(3), a $(1,1)$-LA$(N;k,v)$ is equivalent to a disjoint partial spread system on some ${\mathcal S} \subseteq {\mathcal P}(X)$ that consists of $k$ spreads (note that the condition in Proposition \ref{turntosets} that no two of the $kv$ classes are equal implies in particular that at most one of the $kv$ classes is the empty set). Such a spread system exists if and only if there exists a realizable $v$-type on $\{1,\ldots,N\}$ consisting of $k$ shapes.
    \item
By Proposition \ref{turntosets}(4), a $(\bar{1},1)$-LA$(N;k,v)$ is equivalent to a disjoint partial spread system on some ${\mathcal S} \subseteq {\mathcal P}(X) \setminus \{\emptyset\}$ that consists of $k$ spreads. Such a spread system exists if and only if there exists a realizable $v$-type on $\{1,\ldots,N\}$ consisting of $k$ shapes such that no shape contains a 0.
\end{enumerate}
\end{proof}

We show that admissible types are realizable.

\begin{lemma}\label{full}
If every full admissible type is realizable, then every admissible type is realizable.
\end{lemma}
\begin{proof}
If  $\{M_1,\dots,M_k\}$ is not full, adjoining $M_{k+1} = \{\ell\}$ if
$\binom{N}{\ell} > \sum_{\gamma=1}^k\mu_{M_\gamma}(\ell)$
produces an admissible type; if $\{M_1,\dots,M_{k+1}\}$ is realizable, so is $\{M_1,\dots,M_k\}$.
Iterating this, we can restrict our attention to the case when $\{M_1,\dots,M_k\}$ is full.
\end{proof}
Let ${\cal M} = \{M_1,\dots,M_k\}$ be an admissible full type for the set $X=\{1,\ldots,N\}$,
in which $M_\gamma$ is the shape $\{m_{\gamma,1}, \dots, m_{\gamma,p_\gamma}\}$ for $1 \leq \gamma \leq k$.
A realization of ${\cal M}$ would consist of spreads $\{S_{\gamma,1}, \dots, S_{\gamma,p_\gamma}\}$ for $1 \leq \gamma \leq k$ where $|S_{\gamma,i}|=m_{\gamma,i}$ for $1 \leq i \leq p_{\gamma}$ and $1 \leq \gamma \leq k$. We construct ${\cal M}$ by sequentially assigning the elements $1,\ldots,N$ of $X$ to the sets $S_{\gamma,i}$.
Let $X_\tau = \{1,\dots,\tau\}$ for $0 \leq \tau \leq N$.
The next definition describes the
assignments of the elements of $X_\tau$ to the sets $S_{\gamma,i}$ that obey the obvious necessary condition to be completable to a realization of ${\cal M}$.
A {\em $\tau$-realization} of ${\cal M}$ is a collection of sets $\{ S_{\gamma,j}^\tau \subseteq X_\tau : 1 \leq \gamma \leq k, 1 \leq j \leq p_\gamma\}$ such that for every subset $S \subseteq X_\tau$, and for every $0 \leq \ell \leq N$,
\[|\{(\gamma,j):S_{\gamma,j}^\tau=S \mbox{ and } m_{\gamma,j} = \ell\}|=\tbinom{N-\tau}{\ell-|S|}.\]

When $(M_1,\dots,M_k)$ is an admissible full type, setting $S_{\gamma,j}^0 = \emptyset$ for $1 \leq \gamma \leq k, 1 \leq j \leq p_\gamma$ gives a 0-realization.
We adapt an elegant method of Brouwer and Schrijver \cite{BrouwerS} to show:
\begin{lemma}\label{up1}
If the full admissible type ${\cal M}= \{M_1,\dots,M_k\}$ has a $\tau$-realization and $\tau < N$, then ${\cal M}$ has a $(\tau+1)$-realization.
\end{lemma}
\begin{proof}
Counting all occurrences of elements of $X$ in a putative realization of  ${\cal M}$, \[ T = \medop\sum_{\gamma=1}^k \medop\sum_{j=1}^{p_\gamma} m_{\gamma,j} = \medop\sum_{\ell=0}^N \ell \tbinom{N}{\ell} = N 2^{N-1}.\]
Let $\chi = k - \frac{T}{N} = k-2^{N-1}$, which is an integer.
Each element of $X$ is to appear in a block of $2^{N-1}$ partial spreads, and so must be omitted in $\chi$ partial spreads.

Suppose that  ${\cal M}$ has the $\tau$-realization $\{ S_{\gamma,j}^\tau \subseteq X_\tau : 1 \leq \gamma \leq k, 1 \leq j \leq p_\gamma\}$.
Form a directed multigraph $D$ with vertices $\{ s,t,x \} \cup \{ y_1,\dots,y_k \} \cup \{ (S,\ell) : S \subseteq X_\tau, |S| \leq \ell \leq N\}$.
The arcs of $D$ are:
\begin{enumerate}
\item $\{ (s,y_\gamma), (y_\gamma,x) : 1 \leq \gamma \leq k\}$;
\item $\{ a_{\gamma,j} = (y_\gamma,(S,\ell)) : 1 \leq \gamma \leq k, 1 \leq j \leq p_\gamma \mbox{ when } S =S_{\gamma,j}^\tau \mbox{ and } m_{\gamma,j} = \ell \}$;
\item $\{ ((S,\ell),t) :  S \subseteq X_\tau \mbox{ and } |S| \leq \ell \leq N\}$; and
\item $\{(x,t), (t,s)\}$.
\end{enumerate}

Now we place a flow $f$ on the arcs of $D$. Set
\[ f(a) = \left \{ \begin{array}{ccl}
1 & \mbox{if} & a = (s,y_\gamma) \mbox{ for } 1 \leq \gamma \leq k;\\[0.2cm]
k-2^{N-1} & \mbox{if} & a = (x,t);\\[0.2cm]
k & \mbox{if} & a = (t,s);\\[0.2cm]
\binom{N-1-\tau}{\ell-1-|S|} & \mbox{if} & a = ((S,\ell),t);\\[0.2cm]
\frac{m_{\gamma,j}-|S_{\gamma,j}^\tau|}{N-\tau} &   \mbox{if} & a = a_{\gamma,j} = (y_\gamma,(S,\ell));\\[0.2cm]
1-\frac{ \sum_{j=1}^{p_\gamma} [ m_{\gamma,j}-|S_{\gamma,j}^\tau|]}{N-\tau} &   \mbox{if} & a =  (y_\gamma,x).
\end{array} \right . \]
Now we verify that flow $f$ is a {\em circulation} (inflow equals outflow  at every vertex).
\begin{description}
\item[Flow at $\boldsymbol{s}$:] Inflow and outflow  both equal $k$.
\item[Flow at $\boldsymbol{t}$:] Outflow  is $k$. Inflow is $k-2^{N-1} + \sum_{S \subseteq X_\tau} \sum_{\ell=|S|}^N \binom{N-1-\tau}{\ell-1-|S|} $.
Outflow minus inflow is   \[
\begin{array}{rcl}  2^{N-1} - \sum_{S \subseteq X_\tau} \sum_{\ell=|S|}^N \binom{N-1-\tau}{\ell-1-|S|} & = & 2^{N-1} - \sum_{\sigma=0}^\tau \binom{\tau}{\sigma} \sum_{\ell=\sigma}^N \binom{N-1-\tau}{\ell-1-\sigma} \\[0.2cm]
& = & 2^{N-1} - 2^{N-1-\tau} \sum_{\sigma=0}^\tau \binom{\tau}{\sigma}  \\[0.2cm]
& = & 2^{N-1} - 2^{N-1}  = 0.  \\[0.2cm]
\end{array} \]

\item[Flow at $\boldsymbol{x}$:] Outflow is $k-2^{N-1}$.  Inflow is \[ \sum_{\gamma=1}^k \left [ 1-\frac{ \sum_{j=1}^{p_\gamma} [ m_{\gamma,j}-|S_{\gamma,j}^\tau|]}{N-\tau} \right ] =
k - \mfrac{T}{N-\tau}+\mfrac{1}{N-\tau}\medop\sum_{\gamma=1}^k   \medop\sum_{j=1}^{p_\gamma} |S_{\gamma,j}^\tau| . \]
Because the collection is a $\tau$-realization,  for every $S \subseteq X_\tau$, there are $\binom{N-\tau}{\ell-|S|}$ pairs $(\gamma,j)$ for which $S_{\gamma,j} = S$ and $\ell = m_{\gamma,j}$.
Hence \[
\begin{array}{rcl}  \sum_{\gamma=1}^k   \sum_{j=1}^{p_\gamma} |S_{\gamma,j}^\tau| &
= & \sum_{\ell=0}^N \sum_{\sigma=0}^{\tau} \sigma \binom{\tau}{\sigma} \binom{N-\tau}{\ell-\sigma}\\[0.2cm]
&= & \sum_{\sigma=0}^{\tau} \sigma \binom{\tau}{\sigma} \sum_{\ell=\sigma}^N  \binom{N-\tau}{\ell-\sigma}\\[0.2cm]
&= &  2^{N-\tau} \sum_{\sigma=0}^{\tau} \sigma \binom{\tau}{\sigma}= \tau 2^{N-1}.
\end{array} \]
Thus the inflow is \[ k - \mfrac{T}{N-\tau}+\frac{1}{N-\tau}\medop\sum_{\gamma=1}^k   \medop\sum_{j=1}^{p_\gamma} |S_{\gamma,j}^\tau| =  k - \mfrac{N2^{N-1}}{N-\tau}+\mfrac{\tau 2^{N-1}}{N-\tau} = k-2^{N-1} . \]
\item[Flow at $\boldsymbol{y_\gamma}$, $\boldsymbol{1 \leq \gamma \leq k}$:] Inflow and outflow at $y_\gamma$ both equal 1.
\item[Flow at $\boldsymbol{(S,\ell)}$,  $\boldsymbol{S \subseteq X_\tau}$, $\boldsymbol{|S| \leq \ell \leq N}$:] Outflow at $(S,\ell)$ is $\binom{N-1-\tau}{\ell-1-|S|}$.
Now $S_{\gamma,j}^\tau=S$ when $\ell = m_{\gamma,j}$ for $\binom{N-\tau}{\ell-|S|}$ choices of $(\gamma,j)$ because we have employed a $\tau$-realization.
Thus the inflow at $(S,\ell)$ is
$\frac{\ell-|S|}{N-\tau} \binom{N-\tau}{\ell-|S|} = \binom{N-1-\tau}{\ell-1-|S|}$.
\end{description}
Hence $f$ is a circulation.
By the integer flow theorem \cite{EdmondsKarp72,FordFulkerson}, there is an integer-valued circulation $g$ in which $|g(a) - f(a)| < 1$ for every arc $a$.
Form the $(\tau+1)$-realization by setting $S_{\gamma,j}^{\tau+1} = S_{\gamma,j}^\tau \cup \{ \tau+1\}$ when $g(a_{\gamma,j}) = 1$, and $S_{\gamma,j}^{\tau+1} = S_{\gamma,j}^\tau$ when $g(a_{\gamma,j}) = 0$.
Because $y_\gamma$ has outflow 1, the sets within a group remain disjoint.
Because $(S,\ell)$ has outflow $\binom{N-1-\tau}{\ell-1-|S|}$, the number of sets in the $(\tau+1)$-realization equal to $S \cup \{\tau+1\}$ is $\binom{N-(\tau+1)}{\ell-(|S|+1)}$, as required.
Because $S_{\gamma,j}^\tau=S$ when $\ell = m_{\gamma,j}$ exactly  $\binom{N-\tau}{\ell-|S|}$ times, and of these exactly  $\binom{N-1-\tau}{\ell-1-|S|}$ have $\tau+1$ adjoined, $S_{\gamma,j}^{\tau+1}=S$ when $\ell = m_{\gamma,j}$ exactly  $\binom{N-\tau}{\ell-|S|} - \binom{N-1-\tau}{\ell-1-|S|} = \binom{N-(\tau+1)}{\ell-|S|}$ times, as required.
\end{proof}

An $N$-realization is a realization. So we have proved:

\begin{theorem}\label{genbar}
Type ${\cal M}$  is realizable if and only if it is admissible.
\end{theorem}
\begin{proof}
Every realizable type is admissible.
So suppose that ${\cal M}$ is admissible.
By Lemma \ref{full} we can assume that $\cal M$ is full.
Then form a 0-realization, and apply Lemma \ref{up1} $N$ times to form an $N$-realization of ${\cal M}$.
\end{proof}

This method is constructive, explicitly producing the disjoint partial spread system of the desired type.
Often a very special case  of Theorem \ref{genbar}  is called Baranyai's theorem:

\begin{corollary}
Let $N$, $u$, and $v$ be positive integers with $uv \leq N$.
Write $\binom{N}{k} = \alpha v + \beta$ with $0 \leq \beta < v$.
Then the set of all $k$-subsets of an $N$-set can be partitioned into $\alpha$ partial spreads each containing $v$ $u$-subsets, and (when $\beta > 0$) one partial spread containing $\beta$.
\end{corollary}

Theorem \ref{genbar} is one of many variants of Baranyai's theorem.
Brouwer and Schrijver \cite{BrouwerS}  survey a broad class of generalizations incorporating results from \cite{Baranyai,Baranyai78,Baranyai79}.
While most such generalizations focus on partitions into blocks of size $u$,  disjoint spread systems with all blocks of size at most $u$ have also been examined \cite{BergeJohnson77,BrouwerT81,Johnson78}.
More recently, Bahmanian \cite{Bahmanian12} develops ``detachment'' techniques for such problems; Theorem \ref{genbar} can also be deduced from the proof of his Theorem~6.4.

\section{Maximal admissible \texorpdfstring{$\boldsymbol{v}$}{}-types}\label{sec:types}

In light of Proposition \ref{typeequiv} and Theorem \ref{genbar}, our task for a given $N$ and $v$ is to find the maximum number of shapes in an admissible $v$-type on a set of size $N$. We first establish an upper bound. We then construct examples of $v$-types that meet this bound.

\begin{lemma}\label{upperBoundsLemma}
Let $N$ and $v$ be integers such that $2 \leq v \leq N+1$, let $f = \lfloor \frac{N+1}{v} \rfloor$, and let $d=(f+1)v-N$.
Let $\mathcal{M}$ be an admissible $v$-type for a set of size $N$.
Then $|\mathcal{M}| \leq \Lambda(N,v)$.
\end{lemma}

\begin{proof}
Each shape in $\mathcal{M}$ is a multiset of $v$ nonnegative integers that sum to $N$, and $N=(f+1)v-d<(f+1)v$. So any shape in $\mathcal{M}$ must contain entries less than $f+1$. We quantify this more precisely as follows.
Define the \emph{defect} of an entry $x$ of a shape in $\mathcal{M}$ to be $\max(f+1-x,0)$ and the defect $\delta(M)$ of a shape $M$ in $\mathcal{M}$ to be the sum of the defects of its entries. Then each shape in $\mathcal{M}$ has defect at least $d$.
We can use this fact to bound $|\mathcal{M}|$ based on the number of times integers with positive defect can appear in shapes in $\mathcal{M}$.

Let $\mathcal{M}'$ be the multiset of all shapes in $\mathcal{M}$ whose smallest entry is at most $f-d+1$, and let $\mathcal{M}''=\mathcal{M} \setminus \mathcal{M}'$.  Then $|\mathcal{M}'| \leq \sum_{i=0}^{f-d+1}\tbinom{N}{i}$ because $\mathcal{M}$ is admissible. Also,
\[d|\mathcal{M}''| \leq \medop\sum_{M \in \mathcal{M}''} \delta(M) \leq \medop\sum_{i=f-d+2}^{f}(f+1-i)\tbinom{N}{i},\]
where the first inequality follows because each shape in $\mathcal{M}''$ has defect at least $d$ and the
second inequality is obtained by summing the defects of the entries of the shapes in $\mathcal{M}''$. The result now follows because $|\mathcal{M}|=|\mathcal{M}'|+|\mathcal{M}''|$.
\end{proof}

In what follows, we often take $f = \lfloor \frac{N+1}{v} \rfloor$ and $d=(f+1)v-N$. Note that $2 \leq d \leq v$ when $N \not\equiv v-1 \mod{v}$ and that $d=v+1$ when $N \equiv v-1 \mod{v}$.
For each $N$ and $v$, we now define a $v$-type that we subsequently show is admissible and meets the bound given by Lemma \ref{upperBoundsLemma}.

\begin{definition}\label{optFamDef}
Let $N$ and $v$ be integers such that $2 \leq v \leq N+1$, let $f = \lfloor \frac{N+1}{v} \rfloor$, and let $d=(f+1)v-N$.
For $0 \leq i \leq f$, define $L_i(N,v)$ to be the unique $v$-shape whose smallest entry is equal to $i$ and whose remaining entries differ by at most $1$.
If $v \geq 3$ and $N \equiv v-1 \pmod{v}$, define $L_*(N,v)$ to be the $v$-shape that has two entries equal to $f-1$, $v-3$ entries equal to $f$ and one entry equal to $f+1$.
Define $\mathcal{L}(N,v)$ to be the $v$-type constructed from an empty multiset as follows.
\begin{enumerate}
    \item
Add $\binom{N}{i}$ copies of $L_i(N,v)$ for $0 \leq i \leq f-2$.
    \item
If $N \not\equiv v-1 \mod{v}$, then add $\binom{N}{f-1}$ copies of $L_{f-1}(N,v)$ and $\lfloor\frac{1}{d}(\tbinom{N}{f}-s)\rfloor$ copies of $L_f(N,v)$, where $s=\sum_{i=f-d+2}^{f-1}(d-f-1+i)\tbinom{N}{i}$.
    \item
If $N \equiv v-1 \mod{v}$, then add $\binom{N}{f-1}-2\lceil\frac{s'}{v+1}\rceil$ copies of $L_{f-1}(N,v)$ and $\lceil\frac{s'}{v+1}\rceil$ copies of $L_*(N,v)$, where $s'=\sum_{i=f-v+1}^{f-2}(v-f+i)\tbinom{N}{i}$.
\end{enumerate}
\end{definition}

It is not immediately apparent that $\lfloor\frac{1}{d}(\tbinom{N}{f}-s)\rfloor$ and $\binom{N}{f}-2\lceil\frac{s'}{v+1}\rceil$ are nonnegative.
We show this in Section \ref{sec:main}, where we also show that the  type $\mathcal{L}(N,v)$ is admissible.
To assist with those tasks we require some elementary inequalities involving binomial coefficients which we establish in Section \ref{sec:bin}.

The size of the type $\mathcal{L}(N,v)$ meets the bound of Lemma \ref{upperBoundsLemma}:

\begin{lemma}\label{meetBoundLemma}
Let $N$ and $v$ be integers such that $2 \leq v \leq N+1$, let $f = \lfloor \frac{N+1}{v} \rfloor$, and let $d=(f+1)v-N$.  Then $|\mathcal{L}(N,v)|=\Lambda(N,v)$.
\end{lemma}

\begin{proof}
If $N \not\equiv v-1 \mod{v}$, then
\begin{align*}
|\mathcal{L}(N,v)| &= \textstyle{\left\lfloor\frac{1}{d}\left(\tbinom{N}{f}-\sum_{i=f-d+2}^{f-1}(d-f-1+i)\tbinom{N}{i}\right)\right\rfloor+\sum_{i=0}^{f-1}\tbinom{N}{i}} \\
&= \textstyle{\left\lfloor\frac{1}{d}\left(\tbinom{N}{f}+\sum_{i=f-d+2}^{f-1}(f+1-i)\tbinom{N}{i}\right)\right\rfloor+\sum_{i=0}^{f-d+1}\tbinom{N}{i}} \\
&= \textstyle{\left\lfloor\frac{1}{d}\sum_{i=f-d+2}^{f}(f+1-i)\tbinom{N}{i}\right\rfloor+\sum_{i=0}^{f-d+1}\tbinom{N}{i}}.
\end{align*}

If $N \equiv v-1 \mod{v}$, then $d=v+1$ and
\begin{align*}
|\mathcal{L}(N,v)| &= \textstyle{\tbinom{N}{f-1}-\left\lceil\frac{1}{v+1}\sum_{i=f-v+1}^{f-2}(v-f+i)\tbinom{N}{i}\right\rceil+\sum_{i=0}^{f-2}\tbinom{N}{i}} \\
&= \textstyle{\tbinom{N}{f-1}+\left\lfloor\frac{1}{v+1}\sum_{i=f-v+1}^{f-2}(f+1-i)\tbinom{N}{i}\right\rfloor+\sum_{i=0}^{f-v}\tbinom{N}{i}} \\
&= \textstyle{\left\lfloor\frac{1}{v+1}\sum_{i=f-v+1}^{f}(f+1-i)\tbinom{N}{i}\right\rfloor+\sum_{i=0}^{f-v}\tbinom{N}{i}},
\end{align*}
where the last equality follows because $\binom{N}{f}=(v-1)\binom{N}{f-1}$ (noting that $f=\frac{N+1}{v}$) and hence $\tbinom{N}{f-1}=\frac{2}{v+1}\tbinom{N}{f-1}+\frac{1}{v+1}\tbinom{N}{f}$.
\end{proof}

We conclude this section by establishing that the type $\mathcal{L}(N,v)$ is well-defined and admissible when $v \in \{2,N,N+1\}$.
This allows us to restrict our attention to cases in which $3 \leq v < N$ for the remainder of the paper.

\begin{lemma}\label{bigvLemma}
Let $N$ and $v$ be integers such that $2 \leq v \leq N+1$. If $v \in \{2,N,N+1\}$, then the type $\mathcal{L}(N,v)$ is well-defined and admissible.
\end{lemma}

\begin{proof}
Applying the definition of $\mathcal{L}(N,v)$ we have the following.
If $v=2$ and $N$ is odd, then $\mathcal{L}(N,v)$ consists of exactly $\binom{N}{i}$ copies of the shape $\{i,N-i\}$ for $0 \leq i \leq \frac{N-1}{2}$.
If $v=2$ and $N$ is even, then $\mathcal{L}(N,v)$ consists of exactly $\binom{N}{i}$ copies of the shape $\{i,N-i\}$ for $0 \leq i \leq \frac{N-2}{2}$ and exactly $\frac{1}{2}\binom{N}{N/2}$ copies of the shape $\{\frac{N}{2},\frac{N}{2}\}$.
If $v = N$ and $v \geq 3$, then $\mathcal{L}(N,v)$ consists of exactly one copy of the $v$-shape $\{0,1,\ldots,1,2\}$.
If $v = N+1$ and $v \geq 3$, then $\mathcal{L}(N,v)$ consists of exactly one copy of the $v$-shape $\{0,1,\ldots,1,1\}$.
In each case $\mathcal{L}(N,v)$ is admissible.
\end{proof}

\section{Binomial inequalities}\label{sec:bin}

In this section, we establish certain inequalities on binomial coefficients, and on sums of binomial coefficients.
Although some appear in the literature (\cite{JohnsonNewmanWinston78}, for example), and some are certainly folklore, we prove them here for completeness.

\begin{lemma}
Let $N$ and $v$ be integers such that $3 \leq v < N$. Then
\begin{enumerate}\label{basicFactsLem}
    \item \label{binomRatioItem}
$\tbinom{N}{a-1} = \tfrac{a}{N-a+1}\tbinom{N}{a}$ for any $0 \leq a \leq N$;
    \item \label{binomRatioCorItem}
$\tbinom{N}{a-i} \leq \left(\tfrac{a}{N-a+1}\right)^i\tbinom{N}{a}$ for any $0 \leq a \leq N$ and $0 \leq i \leq a$;
    \item \label{fPlusOneItem}
$\tbinom{N}{a+1} \geq \frac{N(v-1)}{N+v}\tbinom{N}{a}$ for any $0 \leq a \leq \frac{N}{v}$; and
    \item \label{binomSumItem}
$\sum_{i=0}^{a-1}\tbinom{N}{i} < \tfrac{1}{v-2}\tbinom{N}{a}$ for any $1 \leq a \leq \frac{N}{v}$.
\end{enumerate}
\end{lemma}

\begin{proof}\mbox{}
\begin{enumerate}
    \item
This follows from a simple calculation.
    \item
We have
\[\tbinom{N}{a-i}\big/\tbinom{N}{a}=\medop\prod_{j=1}^i \big(\tbinom{N}{a-j}\big/\tbinom{N}{a-j+1}\big).\]
Because $\binom{N}{a-j}/\binom{N}{a-j+1} \leq \binom{N}{a-1}/\binom{N}{a}$ for $1 \leq j \leq i$, the  result follows from (\ref{binomRatioItem}).
    \item
This follows from (\ref{binomRatioItem}) because $a+1 \leq \frac{N}{v}+1$.
    \item
Note $\tfrac{a}{N-a+1} < \tfrac{1}{v-1}$ because $a \leq \frac{N}{v}$. Thus, using (\ref{binomRatioCorItem}), we have
\[\medop\sum_{i=0}^{a-1}\left(\tbinom{N}{i}\big/\tbinom{N}{a}\right) < \medop\sum_{j=1}^{a-1} \left(\tfrac{1}{v-1}\right)^j < \medop\sum_{j=1}^{\infty} \left(\tfrac{1}{v-1}\right)^j = \tfrac{1}{v-2}. \]
\end{enumerate}
\end{proof}

\begin{lemma}\label{easyUnsatIneq}
Let $N$ and $v$ be integers such that $3 \leq v < N$, and let $a = \lfloor\frac{N}{v}\rfloor$. Then
\[(v-1)\medop\sum_{i=0}^{a-1}\tbinom{N}{i} < \tbinom{N}{a+2}.\]
\end{lemma}

\begin{proof}
If $\frac{N}{2} < v \leq N$, then $a=1$ and it is easy to confirm that the inequality holds. So we may assume that $3 \leq v \leq \frac{N}{2}$ and hence that $N \geq 6$. Furthermore, the inequality holds if $v=3$ and $N \in \{6,\ldots,13\}$, so we may assume that $N \geq 14$ if $v=3$.

By Lemma \ref{basicFactsLem}(\ref{binomSumItem}), $\sum_{i=0}^{a-1}\tbinom{N}{i} < \tfrac{1}{v-2}\tbinom{N}{a}$. By Lemma \ref{basicFactsLem}(\ref{fPlusOneItem}), $\tbinom{N}{a+1} \geq \frac{N(v-1)}{N+v}\tbinom{N}{a}$ and, by Lemma \ref{basicFactsLem}(\ref{binomRatioItem}), $\tbinom{N}{a+2} \geq \frac{Nv-N-v}{N+2v}\tbinom{N}{a+1}$ because $a+2 \leq \frac{N}{v}+2$. Thus it suffices to show that $\tfrac{v-1}{v-2} \leq (\frac{N(v-1)}{N+v})(\frac{Nv-N-v}{N+2v})$ or equivalently that $\tfrac{1}{v-2} \leq \frac{N(Nv-N-v)}{(N+v)(N+2v)}$. This holds when $v=3$ and $N \geq 14$, so we may assume that $v \geq 4$. Because $v \leq \frac{N}{2}$, $(N+v)(N+2v) \leq 3N^2$, so $\frac{N(Nv-N-v)}{(N+v)(N+2v)} \geq \frac{2v-3}{6}$ and it is clear that $\tfrac{1}{v-2} \leq \frac{2v-3}{6}$ for $v \geq 4$.
\end{proof}

\begin{lemma}\label{floorPlusOneUnsatIneqs}
Let $N$, $v$ be integers such that $4 \leq v < N$, and let $a = \lfloor\frac{N}{v}\rfloor$. Then
\[\tfrac{v-2}{2}\tbinom{N}{a} + (v-1)\medop\sum_{i=0}^{a-1}\tbinom{N}{i} < \tbinom{N}{a+1}.\]
\end{lemma}

\begin{proof}
If $\frac{N}{2} < v \leq N-1$, then $a=1$ and it is routine to confirm that the inequality holds, so we may assume that $4 \leq v \leq \frac{N}{2}$ and hence that $N \geq 8$.
Similarly, if  $\frac{N}{3} < v \leq \frac{N}{2}$, then $a=2$ and it is routine to confirm that the inequality holds for $N \geq 8$, so we may assume that $4 \leq v \leq \frac{N}{3}$ and hence that $N \geq 12$.
Furthermore, the inequality holds if $v=4$ and $N \in \{12,\ldots,19\}$, so we may assume that $N \geq 20$ if $v=4$.

By Lemma \ref{basicFactsLem}(\ref{binomSumItem}), $\sum_{i=0}^{a-1}\tbinom{N}{i} < \tfrac{1}{v-2}\tbinom{N}{a}$ and, by Lemma \ref{basicFactsLem}(\ref{fPlusOneItem}), $\tbinom{N}{a+1} \geq \frac{N(v-1)}{N+v}\tbinom{N}{a}$. Thus it suffices to show that $\tfrac{v-2}{2} + \frac{v-1}{v-2} \leq \frac{N(v-1)}{N+v}$. This holds when $v=4$ and $N \geq 20$, so we may assume that $v \geq 5$. Because $v \leq \frac{N}{3}$, $\frac{N(v-1)}{N+v} \geq \frac{3(v-1)}{4}$ and it is clear that $\tfrac{v-2}{2} + \frac{v-1}{v-2} \leq \frac{3(v-1)}{4}$ for all $v \geq 5$.
\end{proof}

\begin{lemma}\label{threePartsLemma}
Let $N$ be an integer such that $N \geq 4$, and let $a = \lfloor\frac{N}{v}\rfloor$.
\begin{enumerate}
    \item
If $N \equiv 0 \mod{3}$, then $\tbinom{N}{a-1}+2\tbinom{N}{a-2} + \tbinom{N}{a-3} < \tbinom{N}{a+1}.$
    \item
If $N \equiv 1 \mod{3}$, then $\tfrac{1}{2}\tbinom{N}{a} + 2\tbinom{N}{a-1} + \tbinom{N}{a-2} < \tbinom{N}{a+1}.$
\end{enumerate}
\end{lemma}

\begin{proof}\mbox{}
\begin{enumerate}
    \item
Because $a=\frac{N}{3}$, by Lemma \ref{basicFactsLem}(\ref{binomRatioItem}), $\tbinom{N}{a+1} = \tfrac{2N}{N+3}\tbinom{N}{a}$, and by Lemma \ref{basicFactsLem}(\ref{binomRatioCorItem}), $\tbinom{N}{a-i} \leq (\tfrac{N}{2N+3})^i\tbinom{N}{a}$ for $i \in \{1,2,3\}$. So
$\tbinom{N}{a+1}-\left(\tbinom{N}{a-1}+2\tbinom{N}{a-2} + \tbinom{N}{a-3}\right) \geq \tbinom{N}{a}\left(\tfrac{2N}{N+3}-\tfrac{N}{2N+3}+2(\tfrac{N}{2N+3})^2+(\tfrac{N}{2N+3})^3\right)>0,$
where the final inequality is routine to verify.
    \item
Now $a=\frac{N-1}{3}$. We have from Lemma \ref{basicFactsLem}(\ref{binomRatioItem}) that $\tbinom{N}{a+1} = \tfrac{2N+1}{N+2}\tbinom{N}{a}$ and from Lemma \ref{basicFactsLem}(\ref{binomRatioCorItem}), $\tbinom{N}{a-i} = (\tfrac{N-1}{2N+4})^i\tbinom{N}{a}$. So
$\tbinom{N}{a+1}-\left(\tfrac{1}{2}\tbinom{N}{a}+2\tbinom{N}{a-1} + \tbinom{N}{a-2}\right) \geq \tbinom{N}{a}\left(\tfrac{2N+1}{N+2}-\tfrac{1}{2}(\tfrac{N-1}{2N+4})+2(\tfrac{N-1}{2N+4})^2+(\tfrac{N-1}{2N+4})^3\right)>0,$
where the final inequality is routine to verify.
\end{enumerate}
\end{proof}

\section{Proof of the main result}\label{sec:main}

Throughout this and the next section we take $s=\sum_{i=f-d+2}^{f-1}(d-f-1+i)\tbinom{N}{i}$ and $s'=\sum_{i=f-v+1}^{f-2}(v-f+i)\tbinom{N}{i}$, as in Definition \ref{optFamDef}. We frequently make use of the fact that, if $N$ and $v$ are integers such that $3 \leq v < N$ and $f = \lfloor \frac{N+1}{v} \rfloor$, then  $f = \lfloor \frac{N}{v} \rfloor$ when $N \not\equiv v-1 \mod{v}$ and $f-1 = \lfloor \frac{N}{v} \rfloor$ when $N \equiv v-1 \mod{v}$. We first establish that the types $\mathcal{L}(N,v)$ are well-defined.

\begin{lemma}\label{wellDefLemma}
Let $N$ and $v$ be integers such that $3 \leq v < N$, let $f = \lfloor \frac{N+1}{v} \rfloor$, and let $d=(f+1)v-N$. Then
\begin{enumerate}
    \item
$\lfloor\frac{1}{d}(\tbinom{N}{f}-s)\rfloor \geq 0$ when $N \not\equiv v-1 \mod{v}$; and
    \item
$\binom{N}{f-1}-2\lceil\frac{s'}{v+1}\rceil \geq 0$ when $N \equiv v-1 \mod{v}$.
\end{enumerate}
\end{lemma}

\begin{proof}
Suppose that $N \not\equiv v-1 \mod{v}$.
Because $d-f-1+i \leq v-2$ for $f-d+2 \leq i \leq f-1$, by Lemma~\ref{basicFactsLem}(\ref{binomSumItem}) it follows that $s < \tbinom{N}{f}$.
Thus $\lfloor\frac{1}{d}(\tbinom{N}{f}-s)\rfloor \geq 0$.

Suppose $N \equiv v-1 \mod{v}$. Then $f \geq 2$ because $N>v$. Now $v-f+i \leq v-2$ for $f-v+1 \leq i \leq f-2$ and so it follows from Lemma~\ref{basicFactsLem}(\ref{binomSumItem}) that $s' < \tbinom{N}{f-1}$. Thus
\[\tbinom{N}{f-1}-2\left\lceil\tfrac{s'}{v+1}\right\rceil \geq \tbinom{N}{f-1}-2\left\lceil\tfrac{1}{v+1}\tbinom{N}{f-1}\right\rceil \geq \tbinom{N}{f-1}-2\left\lceil\tfrac{1}{4}\tbinom{N}{f-1}\right\rceil \geq 0. \]
\end{proof}

Next we show that the types $\mathcal{L}(N,v)$ are admissible.

\begin{lemma}\label{admissibleLemma}
Let $N$ and $v$ be integers such that $3 \leq v < N$. Then $\mathcal{L}(N,v)$ is an admissible $v$-type.
\end{lemma}

\begin{proof}
Throughout this proof we abbreviate $L_i(N,v)$ to $L_i$ and $\mathcal{L}(N,v)$ to $\mathcal{L}$.
Let $f = \lfloor \frac{N+1}{v} \rfloor$ and  $d=(f+1)v-N$. For a shape $L$, recall that $\mu_L(x)$ denotes the number of entries of $L$ equal to $x$.
Let $\sigma_{\mathcal{L}}(x)=\sum_{L \in \mathcal{L}}\mu_{L}(x)$.
We must show that $\sigma_{\mathcal{L}}(x) \leq \binom{N}{x}$ for $0 \leq x \leq N$.  We consider several cases depending on $N$ and $x$.

\begin{description}
\item [$\boldsymbol{N \not\equiv v-1 \mod{v}}$:] Treat subcases as follows.
\begin{description}
    \item[$\boldsymbol{\lceil \frac{N}{2} \rceil < x \leq N}$:]
Each entry of $L_i$ for $0 \leq i \leq f$ is at most $\lceil\frac{N}{2}\rceil$, so $\sigma_{\mathcal{L}}(x) = 0$.
    \item[$\boldsymbol{f+1 < x \leq \lceil \frac{N}{2}\rceil}$:]
Then $\mu_{L_i}(x)\leq v-1$ for $0 \leq i \leq f-1$ and $\mu_{L_f}(x)=0$. By the definition of $\mathcal{L}$, the fact that $\binom{N}{f+2} \leq \binom{N}{x}$,  and  Lemma \ref{easyUnsatIneq}, $\sigma_{\mathcal{L}}(x) \leq \binom{N}{x}$.
    \item [$\boldsymbol{x=f+1}$, $\boldsymbol{v \geq 4}$:]
Then $\mu_{L_i}(f+1) \leq v-1$ for $0 \leq i \leq f-1$.
Further, $\mu_{L_f}(f+1) \leq v-2$ and there are at most $\frac{1}{2}\binom{N}{f}$ copies of $L_f$ in $\mathcal{L}$ (note that $d \geq 2$).
By the definition of $\mathcal{L}$ and  Lemma \ref{floorPlusOneUnsatIneqs}, $\sigma_{\mathcal{L}}(f+1)\leq\binom{N}{f+1}$.
    \item [$\boldsymbol{x=f+1}$, $\boldsymbol{v=3}$, $\boldsymbol{N \equiv 0 \mod{3}}$:]
Then $d=3$, $\mu_{L_i}(f+1) = 0$ for $0 \leq i \leq f-4$, $\mu_{L_f}(f+1) = 0$, $\mu_{L_i}(f+1) = 1$ for $i \in \{f-3,f-1\}$, and $\mu_{L_{f-2}}(f+1)=2$.
By the definition of $\mathcal{L}$ and  Lemma \ref{threePartsLemma}(1), $\sigma_{\mathcal{L}}(f+1)\leq\binom{N}{f+1}$.
    \item [$\boldsymbol{x=f+1}$, $\boldsymbol{v=3}$, $\boldsymbol{N \equiv 1 \mod{3}}$:]
Then $d=2$, $\mu_{L_i}(f+1) = 0$ for $0 \leq i \leq f-3$, $\mu_{L_i}(f+1) = 1$ for $i \in \{f-2,f\}$, and $\mu_{L_{f-1}}(f+1)=2$. By the definition of $\mathcal{L}$ and  Lemma \ref{threePartsLemma}(2), $\sigma_{\mathcal{L}}(f+1)\leq\binom{N}{f+1}$.
    \item [$\boldsymbol{x=f}$:]
Then $d \geq 2$, $\mu_{L_i}(f)=0$ for $0 \leq i \leq f-d+1$, $\mu_{L_i}(f)=d-f-1+i$ for $f-d+2 \leq i \leq f-1$, and $\mu_{L_f}(f)=d$. By the definitions of $\mathcal{L}$ and $s$,  $\sigma_{\mathcal{L}}(f)=s+d\lfloor\frac{1}{d}(\tbinom{N}{f}-s)\rfloor \leq \tbinom{N}{f}$.
    \item[$\boldsymbol{0 \leq x < f}$:]
Then $\mu_{L_i}(x) = 1$ if $i = x$ and $\mu_{L_i}(x) =0$ if $i \neq x$.
So $\sigma_{\mathcal{L}}(x)= \binom{N}{x}$.
\end{description}
\item [$\boldsymbol{N \equiv v-1 \mod{v}}$:] Treat subcases as follows.
\begin{description}
    \item[$\boldsymbol{\lceil \frac{N}{2} \rceil < x \leq N}$:]
Each entry of $L_i$ for $0 \leq i \leq f-1$ and each entry of $L_*$ is at most $\lceil\frac{N}{2}\rceil$, so $\sigma_{\mathcal{L}}(x) = 0$.
    \item[$\boldsymbol{f+1 < x \leq \lceil \frac{N}{2}\rceil}$:]
Then $\mu_{L_i}(x)\leq v-1$ for $0 \leq i \leq f-2$, $\mu_{L_{f-1}}(x)=0$ and $\mu_{L_*}(x)=0$. By the definition of $\mathcal{L}$, the fact that $\binom{N}{f+1} < \binom{N}{x}$,  and  Lemma \ref{easyUnsatIneq}, $\sigma_{\mathcal{L}}(x) \leq \binom{N}{x}$.
    \item [$\boldsymbol{x=f+1}$:]
Then $\mu_{L_{f-1}}(f+1)=0$, $\mu_{L_*}(f+1)=1$ and there are $\lceil\frac{s'}{v+1}\rceil$ copies of $L_*$ in $\mathcal{L}$.
Using the definitions of $\mathcal{L}$ and $s'$,
\[ \begin{array}{rcl} \sigma_{\mathcal{L}}(f+1) & = & \medop\sum_{i=0}^{f-2}\tbinom{N}{i}\mu_{L_i}(f+1)+\left\lceil\medop\sum_{i=f-v+1}^{f-2}\tfrac{v-f+i}{v+1}\tbinom{N}{i}\right\rceil \\
& \leq & \left\lceil(v-1)\medop\sum_{i=0}^{f-2}\tbinom{N}{i}\right\rceil \leq \tbinom{N}{f+1}.
\end{array}\]
The first inequality follows from $\mu_{L_i}(f+1)\leq v-2$ for $i \in \{0,\ldots,f-2\} \setminus \{f-v\}$, $\mu_{L_{f-v}}(f+1)=v-1$, and $\frac{v-f+i}{v+1} < 1$ for $f-v+1 \leq i \leq f-2$.
The second inequality follows from Lemma \ref{easyUnsatIneq}.
    \item [$\boldsymbol{x=f}$:]
Then $\mu_{L_i}(f)=0$ for $0 \leq i \leq f-v$, $\mu_{L_i}(f)=v-f+i$ for $f-v+1 \leq i \leq f-1$, and $\mu_{L_*}(f)=v-3$.
Thus, using the definitions of $\mathcal{L}$ and $s'$,
\[
\begin{array}{rcl} \sigma_{\mathcal{L}}(f)&=&s'+(v-1)\left(\tbinom{N}{f-1}-2\left\lceil\tfrac{s'}{v+1}\right\rceil\right)+(v-3)\left\lceil\tfrac{s'}{v+1}\right\rceil\\[2mm]
&=& (v-1)\tbinom{N}{f-1}+s'-(v+1)\left\lceil\tfrac{s'}{v+1}\right\rceil\\[2mm]
&=& \binom{N}{f}+s'-(v+1)\left\lceil\tfrac{s'}{v+1}\right\rceil.
\end{array} \]
The final equality arises from Lemma \ref{basicFactsLem}(\ref{binomRatioItem}) because $f=\frac{N+1}{v}$.
So $\sigma_{\mathcal{L}}(f) \leq \binom{N}{f}$.
    \item [$\boldsymbol{x=f-1}$:]
Then $\mu_{L_i}(f)=0$ for $0 \leq i \leq f-2$, $\mu_{L_{f-1}}(f-1)=1$, and $\mu_{L_*}(f-1)=2$.
By the definitions of $\mathcal{L}$ and $s'$, $\sigma_{\mathcal{L}}(f-1)=\tbinom{N}{f-1}-2\lceil\frac{s'}{v+1}\rceil+2\lceil\frac{s'}{v+1}\rceil=\tbinom{N}{f-1}$.
    \item[$\boldsymbol{0 \leq x < f-1}$:]
Then $\mu_{L_i}(x) = 1$ if $i = x$, $\mu_{L_i}(x) =0$ if $i \neq x$, and $\mu_{L_*}(x)=0$.
So $\sigma_{\mathcal{L}}(x)= \binom{N}{x}$ using the definition of $\mathcal{L}$.
\end{description}
\end{description}
\end{proof}

We can now prove our main result.

\begin{proof}[of Theorem \ref{mainThm}]
By Proposition \ref{typeequiv}(1), the existence of an $(1,1)$-LA$(N;k,v)$ is equivalent to the existence of a realizable $v$-type on $\{1,\ldots,N\}$ consisting of $k$ shapes, and by Theorem \ref{genbar} this is equivalent to the existence of an admissible $v$-type on $\{1,\ldots,N\}$ consisting of $k$ shapes.
Thus ${\rm LAK}_{(1,1)}(N,v)$ is equal to the maximum number of shapes in an admissible $v$-type on $\{1,\ldots,N\}$.
By Lemma \ref{upperBoundsLemma}, no such $v$-type contains more than $\Lambda(N,v)$ shapes.
By Lemmas \ref{meetBoundLemma}, \ref{bigvLemma}, \ref{wellDefLemma} and \ref{admissibleLemma}, $\mathcal{L}(N,v)$ is an admissible $v$-type consisting of $\Lambda(N,v)$ shapes.
\end{proof}

\section{\texorpdfstring{$\boldsymbol{(\bar{1},1)}$}{}-, \texorpdfstring{$\boldsymbol{(1,\bar{1})}$}{}-, and \texorpdfstring{$\boldsymbol{(\bar{1},\bar{1})}$}{}-locating arrays}\label{sec:variant}

For $N$ rows and $v$ symbols,
denote by ${\rm LAK}_{(\bar{1},1)}(N,v)$ the largest number of columns in a  $(\bar{1},1)$-locating array;
by ${\rm LAK}_{(1,\bar{1})}(N,v)$ the largest number of columns in a  $(1,\bar{1})$-locating array; and
by ${\rm LAK}_{(\bar{1},\bar{1})}(N,v)$ the largest number of columns in a  $(\bar{1},\bar{1})$-locating array.
Fortunately these variants on (1,1)-locating arrays can be treated in the same framework.
First we treat the $(\bar{1},1)$ variant.
Observe that if $v \geq N+1$, then a $(\bar{1},1)$-LA$(N;k,v)$ cannot exist for $k > 0$ because no partition of $\{1,\ldots,N\}$ into $v$ distinct nonempty classes exists.

\begin{theorem}\label{1barThm}
Let $N$ and $v$ be integers such that $2 \leq v \leq N$, let $f = \lfloor \frac{N+1}{v} \rfloor$ and $d=(f+1)v-N$. Then
\begin{itemize}
    \item
${\rm LAK}_{(\bar{1},1)}(N,v)= \Lambda(N,v)$ if $d \geq f+2$ and $\sum_{i=0}^{f}(f+1-i)\binom{N}{i} \equiv x \mod{d}$ for some $x \in \{f+1,\ldots,d-1\}$; and
    \item
${\rm LAK}_{(\bar{1},1)}(N,v)= \Lambda(N,v)-1$ otherwise.
\end{itemize}
\end{theorem}

\begin{proof}
Throughout this proof we abbreviate $L_i(N,v)$ to $L_i$ and $\mathcal{L}(N,v)$ to $\mathcal{L}$.
We call a $v$-type $(\bar{1},1)$-admissible if it is admissible and no shape in it contains a 0.
By Proposition \ref{typeequiv}(2), the existence of an ${\rm LAK}_{(\bar{1},1)}(N,v)$ is equivalent to the existence of a realizable $v$-type on $\{1,\ldots,N\}$ consisting of $k$ shapes such that no shape contains a 0.
By Theorem \ref{genbar} this is equivalent to the existence of a $(\bar{1},1)$-admissible $v$-type on $\{1,\ldots,N\}$ consisting of $k$ shapes.
Thus ${\rm LAK}_{(\bar{1},1)}(N,v)$ is equal to the maximum number of shapes in an $(\bar{1},1)$-admissible $v$-type on $\{1,\ldots,N\}$.

\begin{description}
    \item[$\boldsymbol{d \leq f+1}$:]
A very similar argument to that used in the proof of Lemma \ref{upperBoundsLemma} establishes that
\[{\rm LAK}_{(\bar{1},1)}(N,v) \leq \big\lfloor\tfrac{1}{d}\textstyle{\sum_{i=f-d+2}^{f}(f+1-i)}\tbinom{N}{i}\big\rfloor+\textstyle{\sum_{i=1}^{f-d+1}}\tbinom{N}{i} =\Lambda(N,v)-1.\]
Furthermore, we can obtain a $(\bar{1},1)$-admissible $v$-type on $\{1,\ldots,N\}$ with $\Lambda(N,v)-1$ shapes by removing the shape $L_0$ from $\mathcal{L}(N,v)$.
    \item[$\boldsymbol{d \geq f+2}$:]
Let $x \in \{0,\ldots,d-1\}$ satisfy $\sum_{i=0}^{f}(f+1-i)\binom{N}{i} \equiv x \mod{d}$.
Now
\[
\begin{array}{rcl} \big\lfloor\tfrac{1}{d}\textstyle{\sum_{i=0}^{f}(f+1-i)}\tbinom{N}{i}\big\rfloor& = & \Lambda(N,v), \\
\big\lfloor\tfrac{1}{d}\textstyle{\sum_{i=1}^{f}(f+1-i)}\tbinom{N}{i}\big\rfloor &=&
\left\{
  \begin{array}{ll}
    \Lambda(N,v)-1, & \hbox{if $x \in \{0,\ldots,f\}$;} \\
    \Lambda(N,v), & \hbox{if $x \in \{f+1,\ldots,d-1\}$.}
  \end{array}
\right.
\end{array}\]
A very similar argument to that used in the proof of Lemma \ref{upperBoundsLemma} establishes that
${\rm LAK}_{(\bar{1},1)}(N,v) \leq \big\lfloor\tfrac{1}{d}\textstyle{\sum_{i=1}^{f}(f+1-i)}\tbinom{N}{i}\big\rfloor.$
\begin{description}
    \item[$\boldsymbol{x \in \{0,\ldots,f\}}$:]
Again, we obtain a $(\bar{1},1)$-admissible $v$-type on $\{1,\ldots,N\}$ with $\Lambda(N,v)-1$ shapes by removing the shape $L_0$ from $\mathcal{L}(N,v)$.
    \item[$\boldsymbol{x \in \{f+1,\ldots,d-1\}}$, $\boldsymbol{N \not\equiv v-1 \mod{v}}$:]
We claim that the $v$-type $\mathcal{L}'$ obtained from $\mathcal{L}$ by removing the shape $L_0$ and adding one more shape $L_f$ is $(\bar{1},1)$-admissible.
The shapes $L_0$ and $L_f$ are as follows.
\begin{center}
\begin{tabular}{c||c|c|c}
  $x$ & $0$ & $f$ & $f+1$ \\ \hline
  $\mu_{L_0}(x)$ & $1$ & $d-f-1$ & $v-d+f$ \\
  $\mu_{L_f}(x)$ & $0$ & $d$ & $v-d$
\end{tabular}
\end{center}
Thus, $\sigma_{\mathcal{L}'}(0)=0$, $\sigma_{\mathcal{L}'}(f)=\sigma_{\mathcal{L}}(f)+f+1$ and $\sigma_{\mathcal{L}'}(x) \leq \sigma_{\mathcal{L}}(x)$ for each $x \neq f$.
To show that $\mathcal{L}'$ is $(\bar{1},1)$-admissible, we  show that $\sigma_{\mathcal{L}}(f)+f+1 \leq \binom{N}{f}$.
As in the proof of Lemma \ref{admissibleLemma}, $\sigma_{\mathcal{L}}(f) = s+d\lfloor\tfrac{1}{d}(\tbinom{N}{f}-s)\rfloor$.
Using the definition of $s$,
\[\tbinom{N}{f}-s \equiv \medop\sum_{i=0}^{f}(f+1-i)\tbinom{N}{i} \equiv x \mod{d},\]
and so $\sigma_{\mathcal{L}}(f) = \tbinom{N}{f}-x$. Thus
$\sigma_{\mathcal{L}}(f)+f+1 \leq \tbinom{N}{f}$.
    \item[$\boldsymbol{x \in \{f+1,\ldots,d-1\}}$, $\boldsymbol{N \equiv v-1 \mod{v}}$:]
Then $d=v+1$, so $v \geq f+1$ and $x \in \{f+1,\ldots,v\}$.
Also, $f \geq 2$ because $N \geq 2v-1$.
We claim that the $v$-type $\mathcal{L}'$ obtained from $\mathcal{L}$ by removing the shape $L_0$ and one shape $L_*$ and adding two more shapes $L_{f-1}$ is $(\bar{1},1)$-admissible.
The shapes $L_0$ and $L_*$ and $L_{f-1}$ are as follows.
\begin{center}
\begin{tabular}{c||c|c|c|c}
  $x$ & $0$ & $f-1$ & $f$ & $f+1$ \\ \hline
  $\mu_{L_0}(x)$ & $1$ & $0$ & $v-f$ & $f-1$ \\
  $\mu_{L_*}(x)$ & $0$ & $2$ & $v-3$ & $1$ \\
  $\mu_{L_{f-1}}(x)$ & $0$ & $1$ & $v-1$ & $0$
\end{tabular}
\end{center}
Thus, $\sigma_{\mathcal{L}'}(0)=0$, $\sigma_{\mathcal{L}'}(f)=\sigma_{\mathcal{L}}(f)+f+1$ and $\sigma_{\mathcal{L}'}(x) \leq \sigma_{\mathcal{L}}(x)$ for each $x \neq f$.
To show that $\mathcal{L}'$ is $(\bar{1},1)$-admissible, we  show that $\sigma_{\mathcal{L}}(f)+f+1 \leq \binom{N}{f}$.
As in the proof of Lemma \ref{admissibleLemma}, $\sigma_{\mathcal{L}}(f) = \binom{N}{f}+ s'-(v+1)\lceil\tfrac{s'}{v+1}\rceil$. Using the definition of $s'$,
\[
\begin{array}{rcll}
  s' &\equiv& \textstyle{\sum_{i=0}^{f-2}(v-f+i)\tbinom{N}{i}} & \mod{v+1} \\[1mm]
  &\equiv& \textstyle{\sum_{i=0}^{f}(v-f+i)\tbinom{N}{i}} & \mod{v+1} \\[1mm]
  &\equiv& \textstyle{-\sum_{i=0}^{f}(f+1-i)\tbinom{N}{i}} & \mod{v+1} \\[1mm]
  &\equiv& -x & \mod{v+1}
\end{array}
\]
where the second congruence follows because $\binom{N}{f}=(v-1)\binom{N}{f-1}$ by Lemma \ref{basicFactsLem}(\ref{binomRatioItem}) (noting that $f=\frac{N+1}{v}$) and hence  $v\binom{N}{f}+(v-1)\binom{N}{f-1} \equiv 0 \mod {v+1}$. It follows that $\sigma_{\mathcal{L}}(f) = \tbinom{N}{f}-x$. Thus
$\sigma_{\mathcal{L}}(f)+f+1 \leq \tbinom{N}{f}$.
\end{description}
\end{description}
\end{proof}

To treat  $(1,\bar{1})$ and $(\bar{1},\bar{1})$ variants, we also consider the unique 0-way interaction  $\sqcup$,  the interaction containing no columns at all.
For every $N \times k$ array $A$, $\rho_A(\sqcup) = \{1, \dots, N\}$, the set of all rows.
(By comparison, the set $\emptyset$ containing no interactions has $\rho_A(\emptyset) = \emptyset$.)

\begin{theorem}\label{barvar}
Let $N$ and $v$ be integers such that $2 \leq v \leq N+1$. Then
\begin{enumerate}
    \item
${\rm LAK}_{(1,\bar{1})}(N,v) = \left \{
\begin{array}{lcl}
{\rm LAK}_{(1,1)}(N,v) & \mbox{if} & v \geq 3\\
2^{N-1}-1 & \mbox{if} & v =2
\end{array} \right .$
    \item
${\rm LAK}_{(\bar{1},\bar{1})}(N,v) = {\rm LAK}_{(\bar{1},1)}(N,v)$ when $v \leq N$.
\end{enumerate}
\end{theorem}
\begin{proof}
First note that ${\rm LAK}_{(1,1)}(N,2) = 2^{N-1}$.
It follows from the definitions that every $(1,\bar{1})$-LA$(N,v)$ is a $(1,1)$-LA$(N,v)$ and that every $(\bar{1},\bar{1})$-LA$(N,v)$ is a $(\bar{1},1)$-LA$(N,v)$.

Suppose that $A$ is an array with $N$ rows and $v \geq 2$ symbols such that $\rho_A(T) = \{1,\dots,N\}= \rho_A(\sqcup)$ for some 1-way interaction $T$.
Then there are $v-1$ other 1-way interactions $T_1, \dots, T_{v-1}$ involving the same column as $T$, and $\rho_A(T_i) = \emptyset$ for $1 \leq i \leq v-1$.
Thus $\rho_A(T_1) = \rho(\emptyset)$ and $A$ is not a $(1,\bar{1})$-LA$(N,v)$.
Also, if $v \geq 3$, then $\rho_A(T_1)=\rho_A(T_2)$ and $A$ is not a $(1,1)$-LA$(N,v)$.
It follows that every $(\bar{1},1)$-LA$(N,v)$ is a $(\bar{1},\bar{1})$-LA$(N,v)$ and ${\rm LAK}_{(\bar{1},\bar{1})}(N,v) = {\rm LAK}_{(\bar{1},1)}(N,v)$.
Also, if $v \geq 3$, every $(1,1)$-LA$(N,v)$ is a $(1,\bar{1})$-LA$(N,v)$ and ${\rm LAK}_{(1,\bar{1})}(N,v) = {\rm LAK}_{(1,1)}(N,v)$.

Finally, when $v=2$, a $(1,1)$-LA$(N,v)$ $A$ with $2^{N-1}$ columns necessarily has a 1-way interaction $T$ with $\rho_A(T) = \{1,\dots,N\}$, because there are $2^N$ $1$-way interactions and each must have a distinct subset of $\{1,\ldots,N\}$ as its image under $\rho_A$. Hence no $(1,\bar{1})$-locating array can have $2^{N-1}$ columns. Removing the single column of $A$ containing the 1-way interaction $T$ with $\rho_A(T) = \{1,\dots,N\}$ yields a $(1,\bar{1})$-locating array with $2^{N-1}-1$ columns.
\end{proof}

\section{Concluding remarks}\label{sec:asym}
We close with some remarks on the asymptotic sizes of covering and locating arrays.
In \cite{CMjoco}, covering arrays of strength $t+1$ are used to construct $(1,t)$-locating arrays.
Hence it is natural to compare strength two covering arrays with (1,1)-locating arrays.
For $v$ fixed and $k \rightarrow \infty$, the smallest number of rows in a covering array with strength two is $N = \frac{v}{2} (\log k) (1+\mbox{o}(1))$  \cite{GKV2}.
(In this section, all logarithms are base 2.)
We consider (1,1)-locating arrays, the other three variants being similar.
For $f = \lfloor \frac{N+1}{v} \rfloor$ and $k = {\rm LAK}_{(1,1)}(N,v)$, by Theorem \ref{mainThm},  \[
\frac{1}{v+1} \sum_{i=0}^f \binom{N}{i} - 1 \leq  k \leq \sum_{i=0}^f \binom{N}{i} . \]

The {\em binary entropy function} $H(\ell) = -\ell \log \ell - (1-\ell) \log (1-\ell)$.
According to Ash \cite{Ash65} setting $\varepsilon = \frac{f}{N}$ we have
\[ \frac{2^{H(\varepsilon)N}}{\sqrt{8N \varepsilon (1-\varepsilon)}}  \leq \sum_{i=0}^f \binom{N}{i} \leq 2^{H(\varepsilon)N} . \]
Combine these to get \[
 \frac{2^{H(\varepsilon)N-1}}{(v+1)\sqrt{8N \varepsilon (1-\varepsilon)}} \leq\frac{2^{H(\varepsilon)N}}{(v+1)\sqrt{8N \varepsilon (1-\varepsilon)}}  - 1 \leq k \leq 2^{H(\varepsilon)N} . \]
 Take logarithms to get \[
H(\varepsilon)N -  1 -\log (v+1)  - \frac{1}{2} \log [8N \varepsilon (1-\varepsilon)]  \leq \log k \leq H(\varepsilon)N . \]
As $k \rightarrow \infty$ with $v$ fixed, $\varepsilon \rightarrow \frac{1}{v}$ and $H(\varepsilon) \rightarrow \frac{v\log v - (v-1) \log (v-1)}{v}$.
It  follows that $N = \frac{v}{v\log v - (v-1) \log (v-1)} ( \log k) + O(\log\log k)$.
Hence when $v > 2$, (1,1)-locating arrays employ substantially fewer rows than do covering arrays of strength two.

\section*{Acknowledgments}\label{ackref}
This research was completed while the second author was visiting Arizona State University.
He expresses his sincere thanks to China Scholarship Council for financial support and to the School of Computing, Informatics, and Decision Systems Engineering at Arizona State University for their kind hospitality.

\def\cprime{$'$}

\end{document}